\documentclass[11pt,a4paper]{article}

\usepackage{amsmath, amsfonts, amssymb}
\usepackage{graphicx}
\usepackage[latin1]{inputenc}

\def\be#1\ee{\begin{equation}#1\end{equation}}
\newtheorem{thm}{Theorem}[section]

\newtheorem{prop}[thm]{Proposition}

\newtheorem{cor}[thm]{Corollary}
\newtheorem{rem}[thm]{Remark}

\DeclareMathOperator{\sgn}{sgn}
\DeclareMathOperator{\Var}{Var}
\DeclareMathOperator{\mes}{Leb}

\def\R{\mathbb{R}}
\def\E{\mathbb{E}\,}

\newenvironment{proof}[1][] {\noindent {\bf Proof#1:} }{\hspace*{\fill}$\square$\medskip\par}

\def\toalsur{\stackrel{\textrm{a.s.}}{\longrightarrow}}

\def\CC{{\mathcal C}}
\def\EE{{\mathcal E}}
\newcommand{\eps}{\varepsilon}

\def\W21{{\mathbb W}_{2}^{1}}
\def\WW{{\mathbb W}}

\def \=L{\ {\buildrel\hbox{\scriptsize d }\over =}\ }
\def \toalsur{\ {\buildrel\hbox{\scriptsize a.s. }\over \to}\ }

\def \toL{\ {\buildrel L_1 \over \to}\ }

\def\wf{{\widehat{f}}}
\def\wK{{\widehat{K}}}

\begin{document}

\title{\bf Least Energy Approximation for Processes
\\ with Stationary Increments}
\author{
   Zakhar Kabluchko
   \footnote{M\"unster University, Orl\'eans-Ring 10, 48149 M\"unster, Germany,
   email \ {\tt  zakhar.kabluchko@uni-muenster.de}}
     \and
   Mikhail Lifshits
   \footnote{St.Petersburg State University, Russia, Stary
   Peterhof, Bibliotechnaya pl.,2,
   email {\tt mikhail@lifshits.org} and MAI, Link\"oping University.}
}
\date{\today}

\maketitle

\begin{abstract}
A function $f=f_T$ is called least energy approximation to a function $B$
on the interval $[0,T]$ with penalty $Q$ if it solves the variational problem
\[
  \int_0^T \left[ f'(t)^2 + Q(f(t)-B(t)) \right] dt
  \searrow \min.
\]
For quadratic penalty the least energy approximation can be found explicitly.
If $B$ is a random process with stationary increments, then on large intervals
$f_T$ also is close to a process of the same class and the relation between
the corresponding spectral measures can be found. We show that in a long run
(when $T\to \infty$) the expectation of energy of optimal approximation per unit
of time converges to some limit which we compute explicitly. For Gaussian and L\'evy processes we complete
this result with almost sure and $L^1$ convergence.

As an example, the asymptotic expression of approximation energy is found for
fractional Brownian motion.


\end{abstract}
\vskip 1cm

\noindent
\textbf{2010 AMS Mathematics Subject Classification:}
Primary: 60G10;  Secondary: 60G15, 49J40, 41A00.
\bigskip

\noindent
\textbf{Key words and phrases:}
least energy approximation, Gaussian process, L\'evy process, fractional Brownian motion,
process with stationary increments, taut string, variational calculus.
\vfill

\newpage

\section{Introduction}

Least energy approximations play important role both in pure and applied mathematics.
The most important approximation of this kind is known under the name of
{\it taut string}.

Given a target function $B(\cdot)$ and a nonnegative width function
$r(\cdot)$ defined on a time interval $[0,T]$ the {\it taut string} is
a function $f_T$ providing minimum for the {\it energy functional}
\[
  L_T(h):=\int_0^T h'(t)^2 \, dt
\]
among all absolutely continuous functions $h$ with given starting
and final values and satisfying
\[
   B(t)-r(t) \le h(t)\le B(t)+r(t),\qquad 0\le t \le T.
\]
The same function optimizes (under the same restrictions) the graph
length $\int_0^T \sqrt{1+h'(t)^2} dt$, variation $\int_0^T |h'(t)| dt$, and
other  functionals represented as integrals of a convex function of $h'$.

Taut string is a classical object well known in Variational Calculus, in
Mathematical Statistics, see \cite{Davies}, \cite{Mammen}, and in a broad range
of applications such as image processing, see \cite[Chapter 4, Subsection 4.4]{Scherzer}
or communication theory, see \cite{Setterqvist}.

For the case when a {\it random function} $B(\cdot)$ is approximated, very
few information is available. Lifshits and Setterqvist studied in
\cite{LS} the energy of taut string accompanying Wiener process.

Unfortunately the taut string is rather hard to describe and to compute
explicitly. Therefore, the study of other least energy approximations makes
sense. One possible alternative is to replace the rigid boundary
restrictions by introducing some penalty function that controls the deviation
from the target function.

A function $f=f_T$ is called least energy approximation to a function $B$
on the interval $[0,T]$ {\it with penalty $Q$}, if it solves the variational problem
\[
  \int_0^T \left[ f'(t)^2 + Q(f(t)-B(t)) \right] dt
  \searrow \min .
\]
Notice that this approach is very much in the spirit of interpolation theory
from functional analysis.
The classical taut string can be formally included in this setting by using
time-inhomogeneous penalty
\[
  Q(x,t):=
  \begin{cases}0,& |x|\le r(t),\\
  +\infty, & |x|> r(t).
  \end{cases}
\]

One of the most natural choices for penalty is the quadratic penalty
$Q(y)=y^2$ where one can advance sufficiently far with explicit calculations.

\begin{center}
\begin{figure} [ht]
\begin{center}
\includegraphics[height=0.5\textwidth,width= 0.99\textwidth]{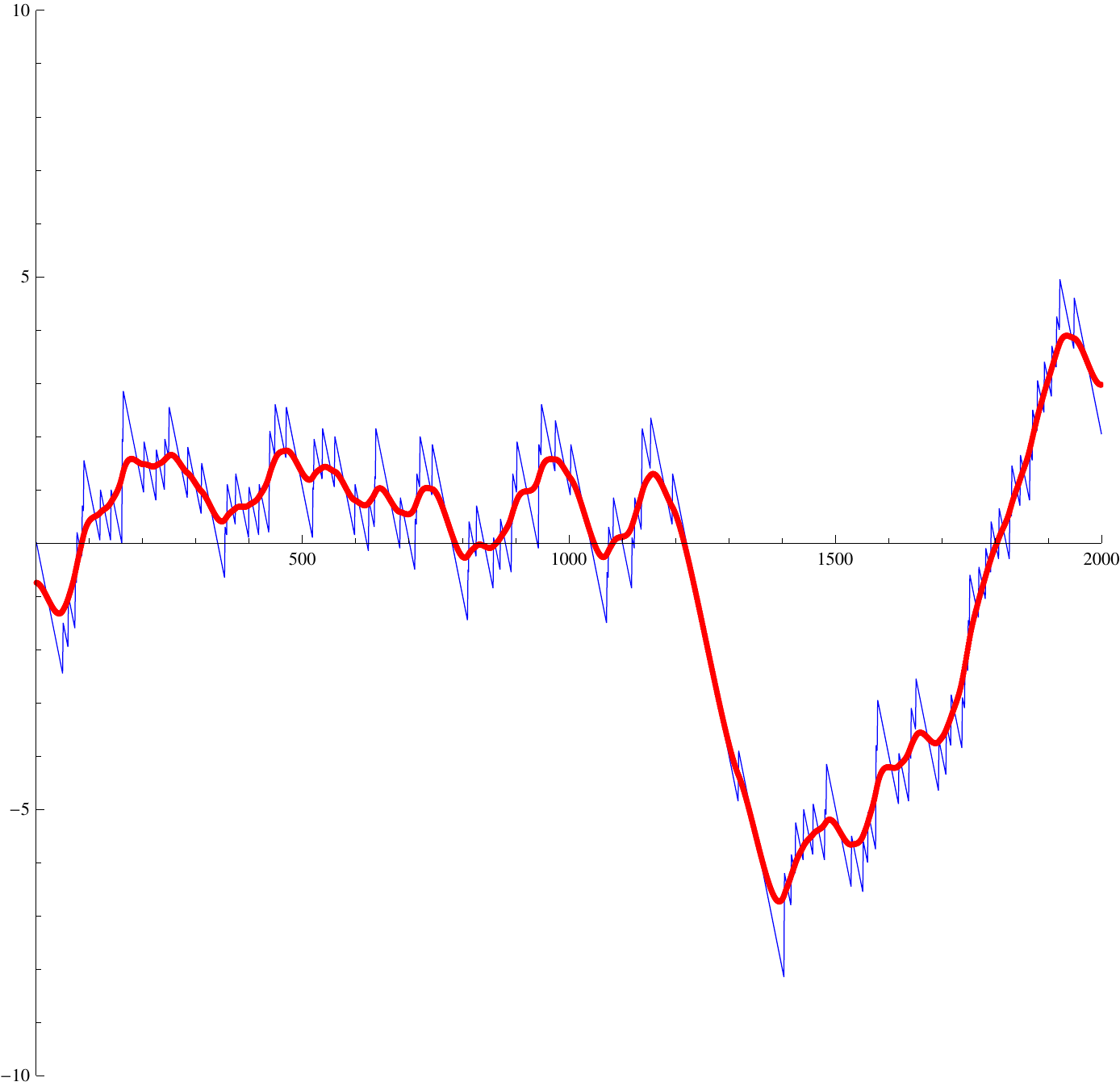}
\end{center}
\caption{The least energy approximation of centered Poisson process.}
\label{fig:lea_Poisson}
\end{figure}
\end{center}


For quadratic penalty the least energy approximation can be found explicitly.
We study its behavior in a long run (when $T\to \infty$) and show that
under weak assumption on $B$ it converges to some limit with exponential rate.
\medskip

In Section \ref{s:determinist} we provide necessary exact and asymptotic
formulas for the least energy approximation of a deterministic function.
In Section \ref{s:statincr} the central results of the article related to
the approximation of a random process with stationary increments
are obtained. If $B$ is a random process with stationary increments, then
on large intervals its least energy approximation
(cf. Figure \ref{fig:lea_Poisson})
also is close to a process of the same class and the relation between the
corresponding spectral measures
can be found. We show that in a long run (when $T\to \infty$) the expectation
of energy of optimal approximation per unit of time converges to some limit.
For Gaussian and L\'evy processes we complete this result with almost sure and
$L^1$ convergence. As an example, the asymptotic expression of approximation
energy is found for fractional Brownian motion. In view of the importance of
Wiener process, we propose an alternative approach to its least energy
approximation in Section \ref{s:Wharm}.
\medskip

Finally, we wish to notice that our results can be considered as a complement
to those of traditional stochastic control theory, where the best approximating
function is chosen among the adapted ones, i.e.\ its value at time $t$ must be
determined by the values of the target function $B(s)$, $s\le t$. Adaptive
approach is more realistic but it leads to the problems solvable mainly for
Markov target processes (see e.g.\ \cite{Kar}, \cite{LS} for the least energy
approximation of Wiener process). One may consider our results as the lower bounds
for the least energy achievable by adaptive control in case it is unknown,
or as the evaluation of price to be paid for not knowing the future, in case
when both optimal adaptive and non-adaptive least energy approximations are known.

\section{Least energy approximation: deterministic setting}
\label{s:determinist}

\subsection{Approximation on a fixed interval}

The following variational problem (which is our starting point)
and its solution are quite standard  facts from variational calculus.
We provide the proof for the sake of completeness but postpone it to
the end of the article.

Let $B(\cdot)$ be a fixed deterministic function on
an interval $[0,T]$. In this section we deal with minimization
problem  of the functional
\[
   \EE_T(f):= \int_0^T \left[ f'(t)^2 + Q(f(t)- B(t)) \right] dt
\]
over the Sobolev space $\W21[0,T]$ of all absolutely continuous functions $f:[0,T]\to \R$
having square integrable derivative.
Here $Q(\cdot)$ is an appropriate penalty function.

\begin{prop}
\label{p:det}
Let $B(\cdot)$ be a measurable function on $[0,T]$ and let $Q(\cdot)$
be a strictly convex non-negative differentiable function on $\R$ such that
\[
  \lim_{x\to\pm\infty} Q(x) =+\infty.
\]
Assume that either $B$ is bounded  or for some $A>0$, $p\ge 0$ it is true that
$B\in L_{p+1}[0,T]$ and
\be \label{Qprime}
   |Q'(x)| \le A (|x|^p +1), \qquad x \in \R.
\ee
Then there exists a unique solution $f_T$ of the problem
\be \label{LQmin}
  \EE_T(f)\searrow \min,  \quad f\in\W21[0,T] \, .
\ee
This solution has absolutely continuous first derivative $f_T'(\cdot)$
and satisfies the equation
\[
   2 \, f_T''(t) = Q'(f_T(t)- B(t)), \qquad  \textrm{for a.e. } t\in (0,T),
 \]
 and the boundary conditions $(f_T)_+'(0)=(f_T)_-'(T)=0$.
\end{prop}

\begin{rem} {\rm In the following we will apply this proposition to random functions
$B$ that are not necessarily bounded but belong to $L_2[0,T]$ almost surely.
Therefore, assumption \eqref{Qprime} is adequate.}
\end{rem}

{\it For the rest of the paper,  we restrict attention to the quadratic penalty $Q(y)=y^2$
because in this case we are able to obtain explicit and quite meaningful results}.

\begin{prop} Let $B\in L_2[0,T]$.
The solution of the minimization problem with quadratic penalty
\[
   \EE_T(f):= \int_0^T \left[ f'(t)^2 + |f(t)- B(t)|^2 \right] dt
   \searrow \min
\]
is given for $0\le t\le T$ by
\be \label{f}
 f_T(t)
 =  - \int_0^t B(s) \sinh(t-s) \, ds
   + \frac{\int_0^T B(s) \cosh(T-s)\, ds} {e^T-e^{-T}}
     \left(e^t+e^{-t}\right).
 \ee
Moreover, for $0< t<T$ we have
 \be \label{fprime}
 f_T'(t)
 = - \int_0^t B(s) \cosh(t-s)\, ds
   + \frac{ \int_0^T B(s) \cosh(T-s)\, ds} {e^T-e^{-T}}
     \left(e^t-e^{-t}\right).
 \ee
\end{prop}

\begin{proof}
From Proposition \ref{p:det} we know that
the minimizer $f_T$ exists uniquely and solves the differential equation
 \be \label{de}
    f''(t) = f(t)- B(t)
 \ee
 with boundary conditions
 \[
    f_{+}'(0)=f_{-}'(T)=0.
 \]
 The general form of the solution for linear equation \eqref{de} without
 boundary conditions is
 \[
   f(t) = f_*(t)+ C_1 e^t +C_2 e^{-t}
 \]
 where $f_*$ is any fixed solution of \eqref{de}.

Let us check that
\[
   f_*(t):= - \int_0^t B(s) \sinh(t-s)ds
   = \frac 12 \left[
   e^{-t} \int_0^t B(s) e^s ds
   -e^{t} \int_0^t B(s) e^{-s} ds
   \right]
\]
indeed provides a solution of \eqref{de}. We have
\begin{eqnarray} \nonumber
   f'_*(t) &=& \frac 12 \left[
   - e^{-t} \int_0^t B(s) e^s ds + e^{-t} B(t) e^{t}
   -e^{t} \int_0^t B(s) e^{-s} ds -e^t B(t) e^{-t}
   \right]
\\ \nonumber
&=&    - \frac 12 \left[ e^{-t} \int_0^t B(s) e^s ds
   +e^{t} \int_0^t B(s) e^{-s} ds \right]
\\ \label{cosh1}
&=&    -
   \int_0^t B(s) \cosh(t-s) ds.
\end{eqnarray}
It follows that
\begin{eqnarray*}
 f''_*(t) &=& - \frac 12 \left[
   - e^{-t} \int_0^t B(s) e^s ds + e^{-t} B(t) e^{t}
   +e^{t} \int_0^t B(s) e^{-s} ds +e^t B(t) e^{-t}
   \right]
\\
&=&    \frac 12 \left[
   e^{-t} \int_0^t B(s) e^s ds
   -e^{t} \int_0^t B(s) e^{-s} ds - 2B(t)
   \right]
\\
&=&   f_*(t)-B(t),
\end{eqnarray*}
and we see that \eqref{de} holds.

Next, we adjust the coefficients $C_1$ and $C_2$ by using boundary
conditions. From
\[
  0= f_+'(0)= (f_{*})_+' (0)+C_1-C_2= C_1-C_2
\]
we derive $C_1=C_2$. From
\[
  0=  f_-'(T) =   (f_{*})_-'(T)+C_1 \, e^T-C_2 \, e^{-T}
   =  (f_{*})_-'(T)+C_1\, (e^T-e^{-T})
\]
by using \eqref{cosh1} we obtain
\[
  C_1= \frac{-   (f_{*})_-'(T)}{e^T-e^{-T}}
     = \frac{ \int_0^T B(s) \cosh(T-s) \, ds}
      {e^T-e^{-T}}
\]
and arrive at formulas \eqref{f} and \eqref{fprime}.
\end{proof}

 \subsection{Approximation in a long run}

We are going now to study the behavior of the least energy approximation
in a long run, i.e.\ when the subject of approximation, function $B(\cdot)$,
is fixed while the interval length $T$ goes to infinity.

In view of future applications it will be more convenient for us
to let $B(\cdot)$ be defined on the entire real line. Although approximation
problem involves only the restriction of $B$ on the positive half-line, if
necessary, one can always extend $B$ to the negative half-line by assigning
it zero values.

Recall that the function $f_T(\cdot)$ defined by
formula \eqref{f} provides the least energy approximation
to $B(\cdot)$ on the interval $[0,T]$.

We first derive a simple approximative heuristics for $f_T$, then transform
this heuristics into a rigorous result.

 \subsubsection{Heuristics}

 We will simplify the expressions for \eqref{f} and \eqref{fprime}
 as follows. Assume that $t,T,T-t,T-s\to +\infty$ and drop all
 small exponential terms $e^{-t}$, $e^{-T}$, $e^{-(T-t)}$, $e^{-(T-s)}$
 where appropriate.

In particular we let
\begin{eqnarray*}
  \frac{\cosh(T-s)} {e^T-e^{-T}} \left(e^t+e^{-t}\right)
  &\approx& \frac 12\left(e^{T-s} +e^{s-T}\right) e^{-T} e^t
  = \frac 12\left(e^{t-s} +e^{s+t-2T}\right)
\\
  &=& \frac 12\left(e^{t-s} +e^{s-t-2(T-t)}\right)
  \approx \frac 12 \ e^{t-s}.
\end{eqnarray*}
By plugging this into \eqref{f} and \eqref{fprime}, we get
 \begin{eqnarray*}
 f_T(t)
&\approx&  \int_0^t B(s) \left(\sinh(s-t) - \frac 12 \ e^{t-s}\right) ds
   +  \int_t^T B(s) \frac 12 \ e^{t-s} ds
 \\
 &=& \frac 12  \int_0^t B(s) \ e^{s-t} ds
     +   \frac 12  \int_t^T B(s) \ e^{t-s} ds
 \\ 
 &=& \frac 12  \int_0^T B(s) \ e^{-|s-t|} ds.
 \end{eqnarray*}
 Similarly,
 \[
  f_T'(t) \approx \frac 12  \int_0^T B(s) \sgn(s-t) \ e^{-|s-t|} ds.
 \]
 Assuming that the function $B(\cdot)$ is defined on
 entire $\R$, it is more natural to use approximations based on "stationary" kernels,
i.e.\  $f_T(t) \approx \wf(t)$ and $f_T'(t) \approx \wf\,'(t)$, where
 \be  \label{fb}
   \wf(t)
   := \frac 12   \int_{-\infty}^\infty B(s) \ e^{-|s-t|} ds
 \ee
 and
 \be  \label{fprime_b}
   \wf\,'(t)
   = \frac 12  \int_{-\infty}^\infty B(s) \sgn(s-t) \ e^{-|s-t|} ds.
 \ee

 Notice that the {\it approximations $\wf$ and $\wf\,'$ do not depend on $T$}. This
 shows a local nature of the least energy approximation in a long run.

 \subsubsection{Rigorous result}

The following result shows that the least energy approximation is
exponentially close to its stationary approximation at the bulk values of
time.

 \begin{prop} \label{p:exp}
 Assume that
 \be \label{Bbound}
   |B(s)| \le C (|s|+1)^p,   \qquad s\in \R,
 \ee
 for some $C,p>0$. Let $f_T$ be the least energy approximation  given by
\eqref{f} and let approximation $\wf$ be given by \eqref{fb}.

Then for all $T\ge 1$ and all
 $t\in[0,T]$ we have
 \be \label{wf_fT}
   |f_T(t)-\wf(t)|\le C\, A_p\, (T+1)^p
   \left( e^{-t}+ e^{-T}+ e^{-(T-t)}\right),
 \ee
 and
 \be \label{wf_fTprime}
   |f'_T(t)-\wf\,'(t)|\le C\, A_p\, (T+1)^p
   \left( e^{-t}+ e^{-T}+ e^{-(T-t)}\right),
 \ee
 where a constant $A_p$ depends only on $p$.
 \end{prop}

\begin{rem} {\rm
Since we are tempted by maximal generality in what concerns $B$,
we will not use this proposition directly in the stochastic part of the article
because of the weak but still a bit restrictive assumption \eqref{Bbound}.
Instead, we will use the elements of its proof later on.}
\end{rem}

 \begin{proof}[ of Proposition \ref{p:exp}]
 Using the definitions \eqref{f} and \eqref{fb} we see that
 \[
    |f_T(t)-\wf(t)| \le I_1+I_2+I_3,
 \]
 where
 \begin{eqnarray*}
  I_1 &=& \frac{1}{2} \int_{-\infty}^0 |B(s)| e^{-(t-s)}\, ds,
  \\
  I_2 &=& \int_0^T |B(s)| \, \left|K_T(s,t)- e^{t-s}/2 \right| \, ds,
  \\
  I_3 &=& \frac{1}{2} \int_T^\infty |B(s)| e^{-(s-t)} \, ds,
 \end{eqnarray*}
 and
 \begin{eqnarray*}
   K_T(s,t)&:=&  \frac{\cosh(T-s)} {e^T-e^{-T}} \left(e^t+e^{-t}\right)
   \\
   &=& \frac{e^{t-s}}{2} \,
   \left(1+ e^{-2(T-s)}\right)
   \left(1+ e^{-2t}\right)
   \left(1- e^{-2T}\right)^{-1}.
  \end{eqnarray*}
 By \eqref{Bbound} we have
 \[
      I_1 \le  \frac{Ce^{-t}}{2} \int_{-\infty}^0 (|s|+1)^p\,  e^{s}\, ds
      := C\, A_p^{(1)}\, e^{-t}
 \]
 and
  \begin{eqnarray*}
     I_3 &\le&  \frac{C e^{-(T-t)}}{2} \int_T^{\infty} (s+1)^p e^{-(s-T)}\, ds
     \\
     &=& \frac{C e^{-(T-t)}}{2} \int_0^{\infty} (T+1+u)^p\, e^{-u}\, du
     \\
     &\le& \frac{C 2^p e^{-(T-t)}}{2} \int_0^{\infty} ((T+1)^p+u^p)\, e^{-u}\, du
     \\
     &\le&  C\, A_p^{(2)} (T+1)^p \ e^{-(T-t)}.
  \end{eqnarray*}
 In order to evaluate $I_2$, notice that
 \[
   K_T(s,t) -\frac{e^{t-s}}{2}
   = \frac{e^{t-s}}{2} \left[ (1+h_1)(1+h_2)(1-h_3)^{-1}-1\right]
 \]
 where  $h_1:=e^{-2(T-s)}$, $h_2:=e^{-2t}$, $h_3:=e^{-2T}$.
 Notice that assumption $T\ge 1$ yields $0\le h_3\le \tfrac 12$\, .
 Therefore, we have
 \[
   (1-h_3)^{-1} = 1+ \frac{h_3}{1-h_3} \le 1+2 h_3.
 \]
 Furthermore,  inequalities  $0\le h_1\le 1$, $0\le h_2\le 1$,
 $0\le h_3\le \tfrac 12$ yield
 \[
   (1+h_1)(1+h_2)(1+2h_3)\le 1+ 4h_1+ 2h_2+ 2h_3.
 \]
 We infer that
 \begin{eqnarray} \nonumber
 0&\le&  K_T(s,t) - \frac{e^{t-s}}{2}
 \\ \nonumber
 &\le& \frac{e^{t-s}}{2} \left( 4h_1+ 2h_2+ 2h_3  \right)
 \\ \nonumber
 &\le& 2 \, e^{t-s} \left( e^{-2(T-s)} + e^{-2t} +e^{-2T}  \right)
  \\ \label{KT_bound}
 &=& 2  \left( e^{-2T+t+s} + e^{-t-s} +e^{t-s-2T}  \right).
 \end{eqnarray}
 It follows that
 \[
    I_2\le 2 \left( I_5+I_6+I_7 \right),
 \]
 where
 \begin{eqnarray*}
   I_5  &:=& C \, e^{-2T+t}  \int_0^T (s+1)^p e^{s} \, ds
   \le  C e^{-(T-t)} (T+1)^p,
 \\
   I_6  &:=& C \, e^{-t} \int_0^T (s+1)^p\, e^{-s}  \, ds
    \le   C\, A_p^{(3)}\, e^{-t},
    \qquad  A_p^{(3)}= e\, \Gamma(p+1),
 \\
   I_7  &:=& C \, e^{t-2T} \int_0^T (s+1)^p\, e^{-s} \, ds
   \le   C\, A_p^{(3)}\, e^{-(T-t)}.
  \end{eqnarray*}
 By combining all estimates we obtain \eqref{wf_fT}.

 The proof of \eqref{wf_fTprime} is completely similar.
One should use the function
\[
  \wK_T(s,t) :=  \frac{\cosh(T-s)} {e^T-e^{-T}} \left(e^t-e^{-t}\right)
\]
instead of $K_T(s,t)$. Then \eqref{KT_bound} is replaced by
\be \label{wKT_bound}
    - e^{-t-s} \le  \wK_T(s,t) - \frac{e^{t-s}}{2} \le  \left( e^{-2T+t+s} +e^{t-s-2T}  \right),
\ee
and all calculations go in the same way.
\end{proof}
 \medskip

 Once the proposition is proved, we have straightforward integral
 estimates. Let $||\cdot||_{2,T}$ denote the norm in the space
 $L_2[0,T]$.

 \begin{cor}
 Under assumptions of Proposition $\ref{p:det}$, we have
\[
   \|f_T-\wf\|_{2,T} \le  2\, C\, A_p\, (T+1)^p
\]
 and
 \be \label{wf_fTprime_int}
   \|f'_T-\wf\,'\|_{2,T} \le 2\, C\, A_p\, (T+1)^p.
 \ee
 \end{cor}

 \begin{proof}
We have
\begin{eqnarray*}
    \|f_T-\wf\|_{2,T}^2
    &=&  \int_0^T \|f_T(t)-\wf(t)|^2 dt
\\
     &\le& C^2\, A_p^2\, (T+1)^{2p}
      \int_0^T \left( e^{-t}+ e^{-T}+ e^{-(T-t)}\right)^2 dt
 \\
     &\le& 3\, C^2\, A_p^2\, (T+1)^{2p}
      \int_0^T \left( e^{-2t}+ e^{-2T}+ e^{-2(T-t)}\right) dt
  \\
     &\le& 3\, C^2\, A_p^2\, (T+1)^{2p}
       \left(\frac 12 + T e^{-2T} + \frac 12\right)
 \\
     &\le& 4\, C^2\, A_p^2\, (T+1)^{2p},
 \end{eqnarray*}
 where at the last step we used assumption $T\ge 1$.

 The proof of \eqref{wf_fTprime_int} is completely similar.
 \end{proof}

 \section{Application to processes with stationary increments}
 \label{s:statincr}

 \subsection{A brief reminder on the processes with stationary increments}

A complex-valued random process $B(t),t\in \R$, is called
{\it process with stationary increments in the wide sense}
if it has finite second moments and the mean and the covariance of the process
$B_{t_0}(t):= B(t_0+t)-B(t_0)$ are the same for all $t_0\in \R$.  If such process is stochastically
continuous, then it admits a spectral representation of the form
 \be \label{BX}
     B(t)= B_0 + D_0 t + \int_{\R\backslash \{0\}} (e^{itu}-1) X(du),
 \ee
where $B_0, D_0$ are random variables with finite second moment and $X(du)$ is a
complex-valued zero mean uncorrelated random measure on $\R\backslash \{0\}$,
uncorrelated with $D_0$, see \cite{Y}. In the following we assume $B_0=0$
since the initial value is irrelevant to our purposes.

Recall that the covariance structure of $B$ is characterized
 by a deterministic spectral measure $\mu_B$ on $\R\backslash\{0\}$ defined by $\mu_B(A):=\E|X(A)|^2$.
 The spectral measure may be infinite but it must satisfy L\'evy's
 integrability condition
 \[
   \int_{\R\backslash \{0\}} \min(u^2,1) \mu_B(du)<\infty.
 \]

In the sequel, $B(t), t\in\R$, denotes a stochastically continuous, real-valued process
with stationary increments in the wide sense such that $B(0)=0$. We always work with
a measurable version of $B$.  Let us stress that even though $B$ is a real-valued process,
the  random measure $X$ may be complex, yet it must satisfy condition $X(-A)=\overline{X(A)}$.

The standard deviation of $B$ grows at most linearly. We will use the following
estimate:
\begin{eqnarray*}
\E|B(t)|^2 &=& \E |D_0|^2 t^2 +  \int_{\R\backslash \{0\}} |e^{itu}-1|^2  \mu_B(du)
\\
&\le& \E |D_0|^2 t^2 +  \int_{-1}^1 u^2  \mu_B(du) t^2 + 4 \mu_B\{\R\backslash [-1,1]\}
\\
&\le&  A (t^2+1),
\end{eqnarray*}
where
\[
   A:= \max\left\{ \E |D_0|^2 + \int_{-1}^1 u^2  \mu_B(du) ,  4 \mu_B\{\R\backslash [-1,1]\} \right\}.
\]
It follows that for any $T>0$
\[
   \E \int_0^T |B(t)|^2 dt <  \infty,
\]
therefore  $B\in L_2[0,T]$ almost surely, and Proposition \ref{p:det} applies to the sample paths
of $B$ with $p=1$.
\medskip

If we additionally assume that the random measure $X$ is Gaussian, then the process
$B$ is also Gaussian.
 {\it Fractional Brownian motions} $W^{(H)}(t)$, $0<H\le 1$, are the most
 known Gaussian processes with stationary increments. The range of parameter $0<H<1$
 is associated to the family of spectral measures
 \[
   \mu_H(du):= \frac{M_H\, du }{|u|^{2H+1}} \, , \qquad
  M_H := \frac{\Gamma(2H+1)\sin(\pi H)}{2\pi}\, ,
 \]
with $D_0=0$. We have a power-type variance
\begin{eqnarray*}
  \E |W^{(H)}(t)|^2 &=& \int_{-\infty}^{\infty} |e^{itu}-1|^2 \frac{M_H\, du}{|u|^{2H+1}}
  \\
  &=& 2 \int_{-\infty}^{\infty} (1-\cos (tu)) \, \frac{M_H\, du}{|u|^{2H+1}}
  \\
  &=& 4 M_H\, |t|^{2H} \int_{0}^{\infty} (1-\cos v)\, \frac{dv}{v^{2H+1}}
  \\
  &=&   4 M_H\, |t|^{2H} (2H)^{-1} \int_{0}^{\infty} \sin v\, \frac{dv}{v^{2H}}
  \\
  &=&  4 M_H\, |t|^{2H} (2H)^{-1}  \ \frac{\pi}{2\Gamma(2H) \sin(\pi H)}
  \\
  &=&  |t|^{2H}.
\end{eqnarray*}
The case $H=1$ is degenerate linear, i.e.\ $W^{(1)}(t)=D_0 t$ with $D_0$ being a standard Gaussian random variable.  The remarkable properties of fractional Brownian motions are described e.g.\ in \cite[Section 7.2]{ST}
and in \cite[Chapter 4]{EM}.

{\it Wiener process} $W$ is a special case of fractional Brownian motion
corresponding to $H=\tfrac 12$. It has a spectral measure
$\mu_W(du)= \tfrac{du}{2\pi |u|^2}$.

\subsection{Convergence of average least energy}

The following result describes the behavior of the
average least energy approximation for arbitrary process with (wide-sense) stationary
increments. We call
\[
   \EE_T(f,B):= \int_0^T \left[ (f(s)-B(s))^2 + f'(s)^2 \right]  ds
\]
the energy of function $f$ with respect to $B$ on the interval $[0,T]$.
If the function $B$ is fixed, we omit it from the notation.

\begin{thm} \label{t:EtoC}
Let $B(t), t\in \R$, be a stochastically continuous process with wide-sense stationary increments
given by its spectral representation \eqref{BX}. Recall that $f_T$
denotes the minimizer of $\EE_T(f,B)$ over all $f\in \W21[0,T]$. Then,
\be
  \lim_{T\to\infty} T^{-1} \E \EE_T(f_T,B) = \CC,
\ee
where
\be \label{CC}
    \CC:= \E|D_0|^2 + \int_{\R\backslash \{0\}} \frac{u^2}{1+u^2} \ \mu_B(du).
\ee
\end{thm}

The constant $\CC$ means the smallest average amount of energy per unit of time
needed for approximation of $B$.

\begin{cor}
For the fractional Brownian motion we have
\begin{eqnarray*}
  \CC &=&  \int_{\R\backslash \{0\}} \frac{u^2}{1+u^2} \ \frac{M_H du}{|u|^{2H+1}}
  = 2 M_H \int_{0}^\infty  \ \frac{ u^{1-2H} \,du}{1+u^2}
\\
  &=&  M_H \int_{0}^\infty  \ \frac{ v^{-H} \,dv}{1+v}
\\
  &=& M_H \cdot \frac{\pi}{\sin(\pi H)} = \frac{\Gamma(2H+1)}{2}.
\end{eqnarray*}
For the Wiener process $H=\tfrac 12$ this yields the constant
\be \label{CW}
   \CC =\frac{1}{2}
\ee
that can be also obtained by other method (see Section $\ref{s:Wharm}$ below).
\end{cor}

\begin{rem} {\rm We can give an alternative {\it non-spectral} representation
for the constant $\CC$, namely,
\[
  \CC= (\E\,D_0)^2 + \frac12 \int_0^{\infty} \Var (B(s))\,  e^{-s}  ds.
\]
Indeed,
\begin{eqnarray*}
 & & (\E D_0)^2 + \frac12 \int_0^{\infty} \Var (B(s))\, e^{-s}  ds
\\
 &=& (\E D_0)^2 + \frac12 \int_0^{\infty} \Var (D_0) \, s^2 e^{-s}  ds
 \\ &&
 + \frac12 \int_0^{\infty}\left(\int_{\R\backslash \{0\}} |1-e^{isu}|^2\ \mu_B(du) \right) e^{-s}  ds
\\
&=& \E D_0^2
 + \frac12 \int_{\R\backslash \{0\}}
   \left( \int_0^{\infty}  (2-e^{isu}- e^{isu}) \, e^{-s}  ds \right) \mu_B(du)
\\
&=& \E D_0^2
 + \frac12 \int_{\R\backslash \{0\}} \left( 2- \frac{1}{1-iu} - \frac{1}{1+iu} \right) \mu_B(du)
\\
&=& \E D_0^2  +  \int_{\R\backslash \{0\}} \frac{u^2}{1+u^2}\  \mu_B(du)  = \CC.
\end{eqnarray*}
}
\end{rem}

\begin{rem} {\rm Would we introduce a viscosity constant $\kappa>0$, i.e.\ change the penalty
to $Q(y)=\kappa^2 y^2$, by a linear time change the minimal energy approximation problem on $[0,T]$
for $B$ reduces to that for the process $\widetilde B(t):= B(t/\kappa)$ on the interval
$[0,\kappa T]$ with $\kappa=1$. Thus the limit energy constant from \eqref{CC} becomes
\[
    \CC_\kappa := \E |D_0|^2 + \int_{\R\backslash \{0\}} \frac{\kappa^2 u^2}{\kappa^2+u^2} \ \mu_B(du)
\]
in this case.
}
\end{rem}

\begin{rem} {\rm It is worthwhile to compare the least energy approximation constant for Wiener process
\eqref{CW} with analogous result of adaptive control, see \cite[p.319]{Kar}. It turns out that the
optimal adaptive control requires {\it two times more energy} than the non-adaptive one.
}
\end{rem}

\begin{proof}[ of Theorem  \ref{t:EtoC}]
Consider approximations
\eqref{fb} and \eqref{fprime_b} related to $B$. They become
\begin{eqnarray*}
   \wf(t) &=& D_0 t +
              \frac 12  \int_{-\infty}^\infty \int_{\R\backslash \{0\}}
    \ e^{-|s-t|}\left(e^{isu}-1\right) X(du)\, ds
   \\ 
   &:=& D_0 t + \int_{\R\backslash \{0\}}  K^{(0)}(t,u) X(du)
\end{eqnarray*}
and
\begin{eqnarray*}
   \wf\,'(t)
   &=&  D_0+ \frac 12  \int_{-\infty}^\infty
   \int_{\R\backslash \{0\}} \sgn(s-t) \ e^{-|s-t|} \left(e^{isu}-1\right)
   X(du) \, ds
   \\ 
   &:=& D_0 + \int_{\R\backslash \{0\}} K^{(1)}(t,u) X(du),
 \end{eqnarray*}
where we have expressions for respective kernels
\begin{eqnarray*}
   K^{(0)}(t,u) &=& \frac 12\, \int_{-\infty}^\infty \ e^{-|s-t|}
  \left(e^{isu}-1\right) ds
  \\
    &=& \frac 12 \ e^{itu} \int_{-\infty}^\infty \ e^{-|v|}e^{ivu} dv -1
  \\
   &=& \frac 12 \ e^{itu} \left(
   \int_{-\infty}^0 \ e^{(1+iu)v}dv  +  \int_{0}^\infty \ e^{- (1-iu)v} dv
   \right) -1
  \\
   &=& \frac 12 \ e^{itu} \left( \frac{1}{1+iu} +  \frac{1}{1-iu} \right) -1
  \\
   &=&  \frac{ e^{itu}}{1+u^2} -1,
 \end{eqnarray*}
 and
 \begin{eqnarray*}
   K^{(1)}(t,u) &=& \frac 12 \int_{-\infty}^\infty \ \sgn(s-t) \, e^{-|s-t|}\,
  \left(e^{isu} -1\right) \, ds
  \\
    &=& \frac 12 \ e^{itu} \int_{-\infty}^\infty \ \sgn(v) \, e^{-|v|}\, e^{ivu}\, dv
  \\
   &=& \frac 12 \ e^{itu} \left(
   - \int_{-\infty}^0 \ e^{(1+iu)v}dv  +  \int_{0}^\infty \ e^{- (1-iu)v}\,  dv
   \right)
  \\
   &=& \frac 12 \ e^{itu} \left( - \frac{1}{1+iu} +  \frac{1}{1-iu} \right)
  \\
   &=&   \frac{iu \, e^{itu}}{1+u^2}\ .
 \end{eqnarray*}
 We conclude that
 \[
   \wf(t)-B(t) =
   \int_{\R\backslash \{0\}} \left( \frac{1}{1+u^2} -1 \right)\, e^{itu} \, X(du)
   = \int_{\R\backslash \{0\}} \frac{-u^2}{1+u^2} \ e^{itu} \, X(du)
\]
and
\be \label{wf_prime}
   \wf\,'(t) = D_0 + \int_{\R\backslash \{0\}} \frac{iu}{1+u^2}\  e^{itu}\, X(du).
\ee
We see that both deviation $\wf-B$ and derivative $\wf\,'$ are {\it wide sense
stationary processes}. By the well known isometric property we have
\[
    \E |\wf(t)-B(t)|^2 = \int_{\R\backslash \{0\}} \frac{u^4}{(1+u^2)^2}\ \mu_B(du),
    \qquad t\in \R,
\]
and
\[
  \E |\wf\,'(t)|^2 = \E|D_0|^2 +\int_{\R\backslash \{0\}} \frac{u^2}{(1+u^2)^2} \ \mu_B(du),
  \qquad t\in \R.
\]
Hence, for any $T>0$ we have
\begin{eqnarray*}
 \E \EE_T(\wf,B) &=& \int_0^T \left[ \E |\wf(t)-B(t)|^2 +  \E |\wf\,'(t)|^2  \right] dt
 \\
 &=&  T \left[ \E |\wf(0)-B(0)|^2 +  \E |\wf\,'(0)|^2  \right]
 \\
 &=&  T  \left[ \E|D_0|^2  + \int_{\R\backslash \{0\}} \frac{u^4+u^2}{(1+u^2)^2}\ \mu_B(du)
     \right]
\\
    &=&  T \left[ \E|D_0|^2 + \int_{\R\backslash \{0\}} \frac{u^2}{1+u^2} \ \mu_B(du) \right] .
\\
    &=&  T \, \CC.
\end{eqnarray*}
It remains to show that the average energy of the least energy approximation $f_T$
and that of stationary approximation $\wf$ are sufficiently close, i.e.\
\be \label{dif1}
   \lim_{T\to\infty} T^{-1} [ \E \EE_T (f_T,B) - \E \EE_T (\wf,B)] = 0.
\ee
Notice that by triangle inequality in $L_2[0,1]$ we have
\[
   \left|  \sqrt{\EE_T (f_T,B)} -  \sqrt{\EE_T (\wf,B)} \right|
   \le  \sqrt{\EE_T (f_T-\wf,0)},
\]
hence,
\[
   \left| \EE_T (f_T,B) -  \EE_T (\wf,B) \right|
   \le  \EE_T (f_T-\wf,0) + 2\sqrt{\EE_T (\wf,B)} \ \sqrt{\EE_T (f_T-\wf,0)}
\]
and
\begin{eqnarray} \nonumber
 && \left|  \E \EE_T (f_T,B) -  \E \EE_T (\wf,B) \right|
 \le  \E \left| \EE_T (f_T,B) -  \EE_T (\wf,B) \right|
\\  \nonumber
  &\le& \E \EE_T (f_T-\wf,0) + 2 \, \E \left(\sqrt{\EE_T (\wf,B)}\  \sqrt{\EE_T (f_T-\wf,0)}\right)
\\  \nonumber
  &\le& \E \EE_T (f_T-\wf,0) + 2 \, \sqrt{ \E \EE_T (\wf,B) \  \E \EE_T (f_T-\wf,0)}
\\   \label{dif1a}
  &=& \E \EE_T (f_T-\wf,0) + 2 \, \sqrt{ \CC T \, \EE_T (f_T-\wf,0)}.
\end{eqnarray}
We see that
\be \label{dif2}
    \lim_{T\to\infty} T^{-1} [ \E \EE_T (f_T -\wf,0)] = 0
\ee
would imply the desired relation \eqref{dif1}.

We first analyze the potential energy and prove that
\be \label{fT_const}
 \sup_{T>0} \sup_{0\le t\le T} \E|f_T(t)-B(t)|^2  <\infty,
\ee

\be  \label{wf_const}
  \sup_{t\ge 0}  \E|\wf(t)-B(t)|^2  <\infty.
\ee
The part concerning $\wf$ is trivial because  $\wf-B$ is stationary.
Furthermore, let us represent
\[
  f_T(t)-B(t) = \int_0^T R_T(s,t) (B(s)-B(t)) ds
\]
where for $s\in [t,T]$ we have by \eqref{KT_bound}
\begin{eqnarray*}
0 &\le&  R_T(s,t) := K_T(s,t)
\\
&\le& \frac{e^{t-s}}{2}+2  \left( e^{-2T+t+s} + e^{-t-s} +e^{t-s-2T}  \right)
\\
&=& \frac{e^{t-s}}{2}  +2  \left( e^{(t-s)+2(s-T)} + e^{(t-s)-2t} +e^{(t-s)-2T}  \right)
\\
&\le&  7 e^{t-s} =  7 e^{-|t-s|},
\end{eqnarray*}
while for $s\in [0,t]$ we have
\begin{eqnarray*}
0 &\le&  R_T(s,t) := K_T(s,t) - \sinh (t-s)
\\
&=&  \left( K_T(s,t)- \frac{e^{t-s}}{2}  \right) +  \frac{e^{s-t}}{2}
\\
&\le& \
 \left( e^{-2T+t+s} + e^{-t-s} +e^{t-s-2T}  \right)  +  \frac{e^{s-t}}{2}
\\
&=& \
 \left( e^{(s-t)-2(T-t)} + e^{(s-t)-2s} +e^{(s-t)-2(T-(t-s))}  \right)  +  \frac{e^{s-t}}{2}
\\
&\le&  7 e^{s-t} =  7 e^{-|t-s|}.
\end{eqnarray*}
Hence,
\begin{eqnarray*}
   \E|f_T(t)-B(t)|^2 &\le&  \E \left( 7 \int_0^T  e^{-|t-s|} |B(s)-B(t)| ds   \right)^2
\\
 &\le&  c \, \E  \int_0^T  e^{-|t-s|} |B(s)-B(t)|^2 ds
\\
 &\le&  c \, A \int_0^T  e^{-|t-s|} [(s-t)^2+1] ds
\end{eqnarray*}
and \eqref{fT_const} follows. By using also \eqref{wf_const} we obtain
\be \label{ff_const}
 \sup_{T>0} \sup_{0\le t\le T} \E|f_T(t)-\wf(t)|^2  <\infty.
\ee
This estimate is still too crude, and we continue as follows.
\begin{eqnarray*}
  f_T(t)-\wf(t) &=&  - \int_{-\infty}^0  e^{-(t-s)} B(s) ds
\\  && + \int_0^T (K_T(s,t)- e^{t-s} /2)  B(s) ds
     -\int_T^\infty e^{-(s-t)} B(s) ds
\\
&:=& G_1(t) +G_2(t) + G_3(t).
\end{eqnarray*}
For $G_1$ we have
\begin{eqnarray*}
  \E |G_1(t)|^2
  &\le&
  e^{-2t} \, \E \left( \int_{-\infty}^0  e^{s} |B(s)| ds   \right)^2
\\
  &\le&
  e^{-2t} \, \E \left( \int_{-\infty}^0  e^{s} |B(s)|^2 ds   \right)
\\
  &\le&
  e^{-2t} A  \int_{-\infty}^0  e^{s} (|s|^2+1) ds
  := c \, e^{-2t}.
\end{eqnarray*}
For $G_3$ we have
\begin{eqnarray*}
  \E |G_3(t)|^2
  &\le&
  e^{-2(T-t)} \E \left( \int_T^{\infty}  e^{T-s} |B(s)| \,ds   \right)^2
\\
  &\le&
  e^{-2(T-t)} \E \left( \int_T^{\infty}  e^{T-s} |B(s)|^2\,  ds   \right)
\\
  &\le&
  e^{-2(T-t)} A \int_T^{\infty}  e^{T-s}   (s^2+1)  \, ds
\\
  &\le&
  e^{-2(T-t)} A \int_T^{\infty}  e^{T-s} (2(s-T)^2+2T^2+1) \, ds
\\
  &\le&
  c A (T^2+1) \, e^{-2(T-t)}.
\end{eqnarray*}
This estimate is not good enough when $t$ is close to $T$.
This is why we need \eqref{ff_const}.
For $G_2$ we have by \eqref{KT_bound}
\[
    |G_2(t)| \le 2(G_{2,1}(t) +  G_{2,2}(t) +  G_{2,3}(t)),
\]
where
\begin{eqnarray*}
   G_{2,1}(t) &=& \int_0^T e^{t-2T+s}|B(s)| ds,
\\
   G_{2,2}(t) &=& \int_0^T e^{-t-s}|B(s)| ds,
\\
   G_{2,3}(t) &=& \int_0^T e^{t-s-2T}|B(s)| ds.
\end{eqnarray*}
Moreover,
\begin{eqnarray*}
   \E |G_{2,1}(t)|^2 &=& e^{2t-4T}\E\left( \int_0^T e^{s}|B(s)| ds\right)^2
\\
   &\le&  e^{2t-4T}\E\left( e^T \int_0^T e^{s}|B(s)|^2 ds\right)
\\
    &\le&  e^{2t-3T}  A \int_0^T e^{s} (s^2+1) ds
\\
    &\le&  2A \, e^{-2(T-t)} (T^2+1),
\end{eqnarray*}

\begin{eqnarray*}
   \E |G_{2,2}(t)|^2 &=& e^{-2t}\E\left( \int_0^T e^{-s}|B(s)| ds\right)^2
\\
    &\le&  e^{-2t} \E\left( \int_0^T e^{-s}|B(s)|^2 ds\right)
 \\
    &\le&  e^{-2t} A \int_0^T e^{-s}  (s^2+1) ds
 \\
    &\le&  3 A \, e^{-2t},
\end{eqnarray*}

\begin{eqnarray*}
   \E |G_{2,3}(t)|^2 &=& e^{2t-4T} \E\left( \int_0^T e^{-s}|B(s)| ds\right)^2
\\
    &\le& e^{2t-4T} \E\left( \int_0^T e^{-s}|B(s)|^2 ds\right)
\\
    &\le& e^{2t-4T}  \int_0^T e^{-s} (s^2+1)  ds
\\
    &\le& 3 A e^{-2(T-t)}.
\end{eqnarray*}
We summarize the latter calculations as
\be \label{poten_t}
  \E |f_T(t)-\wf(t)|^2 \le c A \left(  e^{-2t} +  e^{-2(T-t)} (T^2+1)   \right).
\ee
Now we proceed with integration over $[0,T]$. By applying \eqref{poten_t} on the interval
$[0,T-3\ln T]$
and \eqref{ff_const} on the interval $[T-3\ln T,T]$ we obtain
\be \label{poten_int}
  \E  \int_0^T | f_T(t)-\wf(t)|^2 dt \le c A (1+ \ln T) = o(T).
\ee
Kinetic energy is studied in the same fashion. By using \eqref{wKT_bound} instead
of \eqref{KT_bound}, we replace \eqref{fT_const}, \eqref{wf_const} with
\[
    \sup_{T>0} \sup_{0\le t\le T} \E|f'_T(t)|^2  <\infty,
\]
\[
  \sup_{t\ge 0}  \E|\wf\,'(t)|^2  <\infty.
\]
Hence may be  \eqref{ff_const} replaced by
\[
   \sup_{T>0} \sup_{0\le t\le T} \E|f_T(t)-\wf(t)|^2  <\infty.
\]
Further, using again \eqref{wKT_bound}  we replace \eqref{poten_t}
with
\[  
  \E |f'_T(t)-\wf\,'(t)|^2 \le c A \left(  e^{-2t} +  e^{-2(T-t)} (T^2+1)   \right).
\]
Integration yields
\be \label{kinet_int}
  \E  \int_0^T | f'_T(t)-\wf\,'(t)|^2 dt \le c A (1+ \ln T).
\ee
By merging  \eqref{poten_int} and  \eqref{kinet_int} we get
\be \label{dif3}
  \E \EE_T (f_T-\wf,0) \le c A (1+ \ln T),
\ee
which is a quantitative version of the remaining relation \eqref{dif2}.
\end{proof}

\subsection{Almost sure and $L_1$ convergence}

The random variable
\begin{eqnarray*}
  Z&:=& |D_0|^2 - \E |D_0|^2 + \CC
\\
  &=& |D_0|^2  +\int_{\R\backslash \{0\} }  \frac{u^2}{1+u^2} \, \mu_B(du)
\end{eqnarray*}
is a right candidate for the a.s.\ least energy limit. Notice that if $B$
has no systematic drift (i.e.\ $D_0=0$), then $Z=\CC$ is a deterministic constant.

We first develop a reduction tool showing that almost sure and $L_1$ convergence
of the least approximation energy are reduced to the stationary case.

\begin{prop} \label{p:reduction}
If $ T^{-1} \EE_T(\wf,B) \toalsur Z$, then $ T^{-1} \EE_T(f_T,B) \toalsur Z$.

If $T^{-1} \EE_T(\wf,B) \toL Z$,  then $T^{-1} \EE_T(f_T,B) \toL Z$.
\end{prop}

\begin{proof}
We know from \eqref{dif1a} and  \eqref{dif3} that for any $T$
\be \label{dif4}
  \E  \left|   \EE_T (f_T,B) -  \EE_T (\wf,B) \right|
  \le  c A \sqrt{ T(1+ \ln T)}.
\ee
Let us fix any $a>1$ and consider geometric sequence of times
$T_n:= a^n$.
It follows from  \eqref{dif4} that
\[
  \E  \sum_{n=1}^\infty T_n^{-1}  \left|   \EE_{T_n} (f_{T_n},B) -  \EE_{T_n} (\wf,B) \right|
  < \infty.
\]
Therefore,
\[
   T_n^{-1}  \left|   \EE_{T_n} (f_{T_n},B) -  \EE_{T_n} (\wf,B) \right| \to 0
\]
almost surely and in $L_1$. Under assumptions of our proposition we obtain
\[
   T_n^{-1}  \EE_{T_n} (f_{T_n},B) \to Z
\]
almost surely (resp. in $L_1$).

Furthermore, since the least approximation energy $ \EE_T (f_T,B) $ is an increasing function
of $T$, we have for any $T\in [T_n,T_{n+1}]$
\begin{eqnarray*}
   && \left|  T^{-1}   \EE_{T} (f_{T},B)-Z  \right|
  \\
    &\le&
    \max\left\{  \left|  T^{-1}   \EE_{T_n} (f_{T_n},B)-Z  \right|,
                   \left|  T^{-1}   \EE_{T_{n+1}} (f_{T_{n+1}},B)-Z  \right|
     \right\}
   \\
   &\le&
    a \max\left\{  \left|  T_n^{-1}   \EE_{T_n} (f_{T_n},B)-Z  \right|,
                   \left|  T_{n+1}^{-1}   \EE_{T_{n+1}} (f_{T_{n+1}},B)-Z  \right|
     \right\} +(a-1)Z.
\end{eqnarray*}
Now the a.s. convergence  (resp. $L_1$-convergence) of $T^{-1} \EE_{T} (f_{T},B)$ to $Z$
follows from  that of   $T_n^{-1} \EE_{T_n} (f_{T_n},B)$ by letting $a \searrow 1$.
\end{proof}

\begin{prop}
If the process $B$ is Gaussian and its spectral measure $\mu_B$ has no atoms,
then $T^{-1} \EE_T(f_T,B) \to Z$ in $L_1$ and almost surely,  as $T\to\infty$.
\end{prop}

\begin{proof}
According to Proposition \ref{p:reduction}, it is sufficient to prove that
\[
  T^{-1} \EE_T(\wf,B) \to Z,  \qquad \textrm{as } T\to\infty,
\]
in $L_1$ and almost surely. We may split the energy into potential and kinetic parts
and check that
\[
  T^{-1} \int_0^T |\wf(t)-B(t)|^2 dt \to \int_{\R\backslash \{0\}} \frac{u^4}{(1+u^2)^2}\ \mu_B(du)
\]
and
\[
  T^{-1} \int_0^T   |\wf\,'(t)|^2 dt \to |D_0|^2 +\int_{\R\backslash \{0\}} \frac{u^2}{(1+u^2)^2} \ \mu_B(du).
\]
Recall that by \eqref{wf_prime} we have a representation
\[
    \wf\,'(t) = D_0 + Y(t),
\]
where $Y$ is a centered stationary Gaussian process with the spectral measure
\[
    \nu(du)=\frac{u^2}{(1+u^2)^2}\, \mu_B(du).
\]
It follows that
\[
  T^{-1} \int_0^T   |\wf\,'(t)|^2 dt
  = |D_0|^2 +  2 T^{-1} D_0  \int_0^T   Y(t) dt +  T^{-1} \int_0^T   |Y(t)|^2 dt.
\]

Recall that Gaussian stationary processes whose spectral measure has no atoms are ergodic, see \cite{Gren}, \cite{Mar}.
Since both $B(t)-\wf(t)$ and $Y(t)$ belong to this class, we may use Birkhoff ergodic theorem
and obtain convergence of time-averages to the corresponding expectations in any $L_p$, $p\in (1,\infty)$,
and almost surely, i.e.\
\begin{eqnarray*}
      T^{-1} \int_0^T |\wf(t)-B(t)|^2 dt &\to& \int_{\R\backslash \{0\}} \frac{u^4}{(1+u^2)^2}\ \mu_B(du),
\\
      T^{-1} \int_0^T   Y(t) dt &\to& 0,
\\
      T^{-1} \int_0^T   |Y(t)|^2 dt &\to& \int_{\R\backslash \{0\}} \frac{u^2}{(1+u^2)^2} \ \mu_B(du),
\\
\end{eqnarray*}
and we are done.
\end{proof}
\bigskip

Another interesting class of examples where we can provide an affirmative
answer for the least energy convergence is given by L\'evy processes  (i.e.\
processes with independent stationary increments).

\begin{prop} 
Let $B(t), t\geq 0$, be L\'evy process
having finite second moments. Then  $T^{-1} \EE_T(f_T,B) \toL Z$, and
$T^{-1} \EE_T(f_T,B) \toalsur Z$,  as $T\to\infty$,
where
\[
   Z=\CC = |\E B(1)|^2 + \frac {\Var B(1)}{2}.
\]
\end{prop}

\begin{proof}
Without loss of generality, we may extend $B$ to the negative half-axis in such a way that $-B(-t), t\leq 0$,
is an independent equidistributed copy of $B(t), t\geq 0$.
Looking at $B$ as a process with stationary increments in the wide sense, we see that its linear part
$D_0 t$ is deterministic, and $\E B(t)= D_0 t$, while the spectral measure $\mu_B$
is the same as that of Wiener process up to the numerical factor $\Var B(1)$.
It follows that the hypothetic energy limit $Z$ indeed has the form given in
proposition.

Given the two-sided process $B(t), t\in\R$, we may define an associated L\'evy  noise
(an independently scattered homogeneous random measure) $X(du)$ on $\R$ by
\begin{eqnarray*}
X((t_1,t_2])  :=B(t_2)-B(t_1),  \quad t_1 \leq t_2.
\end{eqnarray*}
Elementary calculations show that
\begin{eqnarray*}
 \wf(t) - B(t) &=&  \frac 12 \int_\R e^{-|\tau|} B(t+\tau) d\tau - B(t)
\\
&=&    \frac 12 \int_\R e^{-|t-u|} \sgn(t-u)  X(du).
\end{eqnarray*}
and
\begin{eqnarray*}
 \wf\,'(t) &=&  \frac 12 \int_\R e^{-|\tau|} \sgn(\tau) B(t+\tau) d\tau
\\
&=&    \frac 12 \int_\R e^{-|t-u|}  X(du).
\end{eqnarray*}
Both processes belong to the class of Ornstein--Uhlenbeck processes driven by L\'evy noise.
As such, they are ergodic, cf.\ \cite [Theorem 5]{RZ1}, \cite{RZ2}, and their squares satisfy
almost sure and $L_1$ ergodic theorems. By applying Proposition \ref{p:reduction}
we complete the proof.
\end{proof}

\begin{rem} {\rm
There exist strong invariance principles providing certain closeness between
sample paths of a Wiener process and a L\'evy process. However, they don't
seem to be good enough for the transfer of least energy approximation estimates.
}
\end{rem}

We conclude this series of examples by mentioning an interesting open problem. Let
$B(t),t\ge 0$, be an $\alpha$-stable L\'evy process with $0<\alpha<2$. Then the sample paths of $B$
are locally bounded, thus the least energy approximation problem is perfectly meaningful.
However, since $B(t)$ has infinite second moment for any $t>0$, the results of the
present paper do not apply. We conjecture that in this case the least approximation
energy would not grow in a quasi-deterministic linear way, but rather imitate a stable
subordinator of order $\alpha/2$ with scaling $T^{2/\alpha}$ because every jump of $B$
of large size $r$ would be reflected by a fast accumulation of energy for approximation function of size $r^2$ during
a finite time interval. The handling of such different behavior would be obviously beyond
the size of the present contribution.

\section{Wiener process: alternative approach}
\label{s:Wharm}

In view of the importance of the Wiener process $W:=W^{(1/2)}$ we find it
reasonable to trace an alternative approach to its least energy approximation. We prove that
\begin{equation} \label{CW_again}
   \E \EE_T(f_T,W) \sim \frac T2\, , \quad \textrm{as } T\to \infty,
\end{equation}
as obtained before in \eqref{CW}.
By the scaling property of the Wiener process, the optimization problem
\[
\EE_T(f,W) = \int_0^T \left( |f(t)-W(t)|^2 +f'(t)^2\right) dt \searrow \min, \; f\in \W21[0,T],
\]
is equivalent to the problem
\be \label{prob01}
   \EE_T^*(f,W)  := \int_0^1 \left( T^2 |f(t)-W(t)|^2 +f'(t)^2\right) dt
   \searrow \min, \; f\in \W21[0,1].
\ee
Let $(e_j)_{j\ge 1}$ be the eigenbase of the covariance operator
of $W$ in $L_2[0,1]$ and let $(\gamma_j)_{j\ge 1}$ be the corresponding
eigenvalues. It is well known that
\[
   \gamma_j = \frac{1}{\pi^2(j-1/2)^2}\, , \qquad e_j(t)=\sqrt{2}\, \sin((j-\tfrac 12)\pi t)
   , \qquad j=1,2,\dots.
\]
Denoting by $w_1, w_2,\ldots$ independent Gaussian random variables with $\E w_j=0$ and $\Var w_j = \gamma_j$, we have the Karhunen--Lo\`{e}ve expansion of the Wiener process
$$
W(t) = \sum_{j=1}^{\infty} w_j e_j(t),\quad t\in [0,1].
$$
For any absolutely continuous function $f$  having square integrable derivative
we may write expansions
\[
   f- f(0):=\sum_{j=1}^\infty f_j \, e_j
\]
and
\[
  f' = \sum_{j=1}^\infty   f_j \, e_j'
     = \sum_{j=1}^\infty  \gamma_j^{-1/2} f_j \,  \gamma_j^{1/2}  e'_j.
\]
Since the system of functions $( \gamma_j^{1/2}  e'_j)_{j\ge 1}$ is orthonormal,
we have
\[
   \int_0^1 f'(t)^2 dt =  \sum_{j=1}^\infty \gamma_j^{-1}  f_j^2\, .
\]
It follows that the problem \eqref{prob01} takes a coordinate form
\begin{multline}
T^2 \int_0^1 \left| f(0) + \sum_{j=1}^\infty (f_j-w_j)e_j\right|^2 dt + \sum_{j=1}^{\infty}\gamma_j^{-1}  f_j^2
\\
= T^2 \left( f(0)^2+2 f(0) S +  \sum_{j=1}^\infty  |f_j-w_j|^2\right)
     + \sum_{j=1}^\infty \gamma_j^{-1}  f_j^2
    \searrow    \min,
\end{multline}
where
\begin{equation}\label{eq:def_S}
S := \sum_{j=1}^\infty (f_j-w_j) \sqrt{2\gamma_j}.
\end{equation}

We first optimize over initial value $f(0)$ with other coordinates $f_j$ fixed and find
$f(0)=-S$. Now, the optimization problem becomes
\be \label{w_energy}
    \EE_T^*(f,W)=
    -T^2 S^2
    +\sum_{j=1}^{\infty}  \left( \left(T^2+ \frac{1}{\gamma_j}\right) f_j^2 -2 T^2 f_j w_j  +T^2 w_j^2 \right)
    \searrow    \min.
\ee
Since partial derivatives in each $f_j$ must vanish, we find
\be \label{fj}
  f_j= T^2 \ \frac{S\sqrt{2\gamma_j}+w_j}{T^2+\frac{1}{\gamma_j}}, \quad j=1,2,\ldots.
\ee
It follows that
\[
    \left( T^2+ \frac{1}{\gamma_j} \right) f_j^2
     =
     \frac{T^4( 2S^2 \gamma_j + 2\sqrt{2\gamma_j} S w_j  +w_j^2)}{T^2+\frac{1}{\gamma_j}} \ ,
\]
\[
   2 T^2 f_j w_j
   =
    \frac{ 2 T^4 \sqrt{2\gamma_j}S w_j + 2 T^4 w_j^2}{T^2+\frac{1}{\gamma_j}} \ .
\]
We plug these expressions into \eqref{w_energy}, notice that the terms containing $S w_j$
cancel and obtain the least approximation energy
\begin{eqnarray*}
 \EE_T(f_T,W)
  &=&
 -T^2 S^2 +\sum_{j=1}^{\infty}    \frac{2T^4\gamma_j} {T^2+\frac{1}{\gamma_j}}\  S^2
 +  \sum_{j=1}^{\infty}  \left( T^2 -   \frac{T^4}
{T^2+\frac{1}{\gamma_j}}    \right)  w_j^2 \, .
\end{eqnarray*}
By using the identity
\be \label{sumgam}
    \sum_{j=1}^{\infty} (2\gamma_j) = 1
\ee
we may rewrite the least energy as
\begin{eqnarray*}
   \EE_T(f_T,W) &=&
   - \sum_{j=1}^{\infty}  \frac{2T^2} {T^2+\frac{1}{\gamma_j}} \  S^2
   +\sum_{j=1}^{\infty}      \frac{T^2 \, \frac {1}{\gamma_j}} {T^2+\frac{1}{\gamma_j}}\  w_j^2 \, .
 \end{eqnarray*}
Recall that $\E w_j^2= \gamma_j$. We find the average least energy
\[
 \E \EE_T(f_T,W) =  \left( -2 \E [S^2]+ 1    \right) \sum_{j=1}^{\infty}  \frac{1} {1+\frac{1}{T^2\gamma_j}} \ .
\]
Since
\be\label{sumint}
   \sum_{j=1}^{\infty}  \frac{1} {1+\frac{1}{T^2\gamma_j}} \sim \frac{T}{\pi}\ \int_0^\infty \frac{dx}{1+x^2}
   = \frac{T}{2}, \quad \text{as } T\to\infty,
\ee
and, as we will see,
\be \label{S2}
    \E S^2 \sim  \frac{1}{T}, \quad \text{as } T\to\infty,
\ee
we arrive at~\eqref{CW_again}.

It remains to analyze the behavior of $S$. By using the definition of $S$ in~\eqref{eq:def_S} and formulae
\eqref{fj} we get an equation
\[
   S = \sum_{j=1}^\infty \sqrt{2\gamma_j}
       \left(  T^2 \ \frac{S\sqrt{2\gamma_j}+w_j}{T^2+\frac{1}{\gamma_j}} - w_j  \right).
\]
Solving it in $S$ we get
\[
   S=\frac{A(T)}{D(T)}
\]
with
\begin{eqnarray*}
  A(T)&=& \sum_{j=1}^\infty \frac{w_j \sqrt{2\gamma_j} }{T^2\gamma_j+1}\ ,
  \\
   D(T)&=& \left(\sum_{j=1}^\infty \frac{2\gamma_j}{1+\frac{1}{T^2\gamma_j}} \right)-1.
\end{eqnarray*}
Notice that $A(T)$ is a centered normal random variable with variance
\[
    \E |A(T)|^2 =  \sum_{j=1}^\infty \frac{ 2\gamma_j^2} {(T^2\gamma_j+1)^2}
    \sim \frac{2}{\pi} \int_0^\infty \frac{dx}{(1+x^2)^2}\ T^{-3} \sim \frac 1 {T^3}
\]
and for the non-random denominator $D(T)$ by using \eqref{sumgam} and \eqref{sumint}
we have
\[
  D(T) = - \frac{2}{T^2} \sum_{j=1}^\infty \frac{1}{1+\frac{1}{T^2\gamma_j}}
  \sim  - \frac{1}{T} \,  .
\]
Now \eqref{S2} is confirmed and we are done.

\section{Addendum: Proof of Proposition \ref{p:det}}

\begin{proof}[ of Proposition \ref{p:det}]
We first notice that
\[
   \inf_{f\in \W21[0,T]} \EE_T(f)\le  \EE_T(0)= \int_0^T Q(-B(t))\, dt < \infty.
\]
Indeed, if $B(\cdot)$ is bounded, then $Q(-B(\cdot))$ is bounded.
If \eqref{Qprime} holds, then we have
\[
  Q(x)\le \widetilde{A}\, (|x|^{p+1}+1), \qquad x\in \R,
\]
with  $\widetilde{A} := 2A+Q(0)$. Hence, if $B\in L_{p+1}[0,T]$, then
\[
  \int_0^T Q(-B(t))\, dt \le \widetilde{A} \int_0^T (|B(t)|^{p+1}+1) \, dt <\infty.
\]

Second, we show that we may restrict the minimization in \eqref{LQmin} on
a subclass
\be \label{inWW}
   \WW_{C,M} :=\left\{ f\in \W21[0,T]:\, \int_0^T f'(t)^2dt
   \le C, \max_{0\le t\le T} |f(t)| \le M\right\}
\ee
with sufficiently large parameters $C$ and $M$.
To justify this statement, it is sufficient to show that for each $r>0$ there
exist $C,M$ such that
\[
  \left\{ f: \EE_T(f) \le r   \right\} \subset \WW_{C,M}.
\]
Let $f$ be such that $\EE_T(f) \le r$. Then
\[
   \int_0^T f'(t)^2 dt \le \EE_T(f) \le r
\]
and we may let $C:=r$. Moreover, for any $0\le s,t\le T$ we have
$|f(s)-f(t)|\le \sqrt{rT}$ by using H\"older inequality. Therefore,
if $|f(s)|\ge M$ holds for some $s\in[0,T]$, then $|f(t)|\ge M- \sqrt{rT}$ for
all $t\in[0,T]$. We see that
\begin{eqnarray*}
r &\ge& \EE_T (f)\ge \int_{\{t: |B(t)|\le (M-\sqrt{rT})/2\}} Q(f(t)-B(t))\, dt
\\
&\ge& \mes\{t: |B(t)|\le (M-\sqrt{rT})/2\}\ \inf\{Q(x): |x|\ge (M-\sqrt{rT})/2\}.
\end{eqnarray*}
When $M$ goes to infinity, then the first term tends to $T$, while the second
tends to infinity due to the assumption $\lim_{x\to\pm \infty}Q(x)=+\infty$.
Therefore for large $M$ we obtain a contradiction. Hence, for such $M$
assumption $|f(s)|\ge M$ cannot hold, and \eqref{inWW} is confirmed.

Next, we show that the minimum of the problem \eqref{LQmin} is attained.
Since the functional $\EE_T(\cdot)$ is lower semi-continuous with respect
to the uniform convergence (notice that the potential part of the energy is even continuous),
and since $\WW_{C,M}$ is relatively compact
with respect to the topology of uniform convergence, the minimum of $\EE_T(\cdot)$
on  $\WW_{C,M}$  is indeed attained on some set of minimizers.

Next, since $Q(\cdot)$ is strictly convex, the functional
$\EE_T(\cdot)$ is also strictly convex, hence the minimizer is unique. Let us denote it $f_T$.

By Lebesgue theorem, G\^{a}teaux (directional) derivative of $\EE_T(\cdot)$
\begin{eqnarray*}
    \delta \EE_T(f_T,\eps)
    &:=& \lim_{h\to 0} h^{-1} \left( \EE_T(f_T+h \eps) - \EE_T(f_T) \right)
\\
    &=&   \int_0^T \left[ 2f_T'(t)\eps'(t) + Q'(f_T(t)- B(t))\eps(t) \right] dt
\end{eqnarray*}
is well defined and must vanish at $f_T$ for any $\eps\in \W21[0,T]$. We have thus
\be \label{vanishQ}
   \int_0^T \left[ 2f_T'(t)\eps'(t) + Q'(f_T(t)- B(t))\eps(t) \right] dt = 0.
\ee

\begin{footnotesize}
To justify the application of Lebesgue dominated convergence theorem to the integrals
\[
        \int_0^T   h^{-1}  \left[  Q(f_T(t) +h\eps(t) - B(t))- Q(f_T(t)- B(t))  \right] dt,
\]
notice that both functions $f_T$ and $\eps_T$ are bounded, say,
$\max|f_T(t)|\le M_1$, $\max|\eps(t)|\le M_2$. Therefore, if $B(\cdot)$ is bounded,
say,  $\max|B(t)|\le M_3$,
then for $h\le 1$ the integrand is uniformly bounded by a constant
$M_2 \sup_{|x|\le M_1+M_2+M_3} |Q'(x)|$
and we use the fact that the derivative of a convex function (if it exists everywhere)
is locally bounded.

Alternatively, if \eqref{Qprime} holds, we have
\begin{eqnarray*}
 && h^{-1}  \left|  Q(f_T(t) +h\eps(t) - B(t))- Q(f_T(t)- B(t))  \right|
 \\
 &=&    |Q'(f_T(t) +\theta h\eps(t) - B(t))| \ |\eps(t)|
 \\
 &\le&  A( ( M_1+M_2 + |B(t)| )^p+1) M_2 ,
\end{eqnarray*}
which also provides an integrable majorant due to $B\in L_p[0,T]$.
\end{footnotesize}
\medskip

Let us fix some $\tau\in (0,T]$ and apply \eqref{vanishQ} to the functions
\[
   \eps_{u}(t):=
   \begin{cases}
   1,              & 0\le t\le \tau-u;\\
   0,              & t\ge \tau;\\
   u^{-1}(\tau-t), & \tau-u\le t\le \tau,
   \end{cases}
\]
with $0<u<\tau$. We obtain
\begin{eqnarray*}
   0 &=& \int_0^T \left[ 2f_T'(t)\eps_u'(t) + Q'(f_T(t)- B(t))\eps_u(t) \right] dt
   \\
    &=& - 2u^{-1} \int_{\tau-u}^{\tau} f_T'(t) dt
   \\
     &&  + \int_0^{\tau-u} Q'(f_T(t)- B(t)) \, dt
     + \int_{\tau-u}^{\tau} Q'(f_T(t)- B(t))\eps_u(t) \, dt
   \\
    &=& - 2u^{-1} (f_T(\tau)-f_T(\tau-u))
   \\
     &&  + \int_0^{\tau-u} Q'(f_T(t)- B(t)) \, dt
      + \int_{\tau-u}^{\tau} Q'(f_T(t)- B(t))\eps_u(t) \, dt.
\end{eqnarray*}
Now we let $u\searrow 0$ and see that the left derivative $(f_T)_-'(\tau)$
exists and
\be \label{leftder}
   2 \, (f_T)_-'(\tau) =  \int_0^{\tau} Q'(f_T(t)- B(t)) \, dt.
\ee
In exactly the same way one obtains
\be \label{rightder}
   2 \, (f_T)_+'(\tau) = - \int_\tau^{T} Q'(f_T(t)- B(t)) \, dt,
\ee
Since $f_T'(\cdot)$ exists almost everywhere, for some $\tau\in (0,T)$
we have $(f_T)_-'(\tau)=(f_T)_+'(\tau)$, i.e.\
\be \label{int0}
  \int_0^{T} Q'(f_T(t)- B(t)) \, dt = 0.
\ee
This fact in turn proves that $(f_T)_-'(\tau)=(f_T)_+'(\tau)$ for every
$\tau\in (0,T)$, i.e.\ the function $f_T(\cdot)$ is differentiable
everywhere and we have
\[
   2\, f_T'(\tau) = - \int_0^{\tau} Q'(f_T(t)- B(t)) \, dt,  \qquad \tau\in (0,T).
\]
By \eqref{int0}, the  boundary conditions $(f_T)_+'(0)=(f_T)_-'(T)=0$ also follow
from representations \eqref{leftder} and \eqref{rightder}.
\end{proof}

\section{Concluding remark}

Of course, one would like to handle the case of more or less general
penalty functions $Q$. But in this case the equation replacing \eqref{de}
is not linear and we don't know whether we may proceed with some, may be
inexplicit, analogues of exponential functions. It would be also nice to
guess a stationary approximation for least energy functions related to
general penalty $Q$.

\section*{Acknowledgments}

The work of the second named author was supported by grants NSh.2504.2014.1, RFBR 13-01-00172, and SPbSU 6.38.672.2013.

\section{Compliance with ethical standards}

The authors declare that they have no conflict of interest.



\begin{thebibliography}{99}

{\baselineskip=10pt

\bibitem{Davies}
Davies P.L., Kovac, A. Local extremes, runs, strings and multiresolution.
Ann. Statist., 29, No. 1, 1--65 (2001)

\bibitem{EM}
Embrechts, P., Maejima, M.  Selfsimilar Processes. Princeton University
Press, Princeton (2002)

\bibitem{Gren}
Grenander, U.  Stochastic processes and statistical inference. Arkiv
Mat., 1, 195--277 (1950)

\bibitem{Kar}
Karatzas, I. On a stochastic representation for the principal eigenvalue
of a second order differential equation. Stochastics, 3, 305--321 (1980)

\bibitem{LS}
Lifshits, M., Setterqvist, E.  Energy of taut string
accompanying Wiener process, Stoch. Proc. Appl., 125, 401--427 (2015)

\bibitem{Mammen}
Mammen, E., van de Geer, S. Locally adaptive regression splines. Ann.
Statist., 25, No. 1, 387--413 (1997)

\bibitem{Mar}
Maruyama, G. The harmonic analysis of stationary stochastic processes.
Mem. Fac. Sci. Kyusyu Univ. A4, 45--106 (1949)

\bibitem{RZ1}
Rosinski, J., Zak, T.  Simple conditions for mixing of infinitely
divisible processes. Stoch. Proc. Appl., 61, 277--288 (1996)

\bibitem{RZ2}
Rosinski, J., Zak, T. The equivalence of ergodicity and weak mixing
for infinitely divisible processes, J. Theor. Probab., 10, 73--86
(1997)

\bibitem{ST}
Samorodnitsky, G., Taqqu, M.S. Stable Non-Gaussian Random Processes.
Chapman \& Hall, New York (1994)

\bibitem{Scherzer}
Scherzer, O. et al.
Variational Methods in Imaging, Ser. Applied Mathematical Sciences,
Vol. 167, Springer, New York (2009)

\bibitem{Setterqvist}
Setterqvist, E., Forchheimer, R.  Application of $\varphi$-stable sets
to a buffered real-time communication system. Proceedings of the 10th
Swedish National Computer Networking Workshop (2014)

\bibitem{Y} Yaglom, A.M.  An Introduction to the Theory of Stationary
Random Functions. Revised English edition. Prentice Hall Inc.,
Englewood Cliffs, N.J. (1962)
}
\end{thebibliography}
\end{document}